\documentclass[11pt,a4paper,leq]{amsart}
\usepackage{amssymb}
\usepackage{amsmath}
\usepackage{amssymb}

\def\constr#1^#2{\mathrel{\mathop{\kern 0pt#1}\limits^{#2}}}
\def\build#1_#2{\mathrel{\mathop{\kern 0pt#1}\limits_{#2}}}
  
at 8pt

\newtheorem{thm}{Theorem}[section]

\newtheorem{lem}[thm]{Lemma}
\newtheorem{prop}[thm]{Proposition}
\theoremstyle{definition}
\newtheorem{defn}[thm]{Definition} \theoremstyle{remark}
\newtheorem{rem}[thm]{Remark}
\numberwithin{equation}{section}
\def\tag{\hfill{}}

\begin{document}

\title{Arithmetic of curves over two dimensional local field}
\author{Belgacem DRAOUIL}

\address{D\'{e}partement de Math\'{e}matiques.\\
Facult\'{e} des Sciences de Bizerte \ 7021, Zarzouna Bizerte.}
\email{Belgacem.Draouil@fsb.rnu.tn}


\begin{abstract}
We study the class field theory of curve defined over two
dimensional local field. The approch used here is a combination of
the work of Kato-Saito, and Yoshida where the base field is one
dimensional
\end{abstract}

\maketitle
\section{Introduction}

Let $k_{1}$ be a local field with finite residue field and let $X$
be a proper smooth geometrically irreducible curve over $k_{1}.$ To
study the fundamental group $\pi _{1}^{ab}\left( X\right) $, Saito
in [8], introduced
the groups $SK_{1}\left( X\right) $ and $V(X)$ and construct the maps $%
\sigma :SK_{1}\left( X\right) \longrightarrow \pi _{1}^{ab}\left(
X\right) $
and $\tau $ $:$ $V(X)$ $\longrightarrow \pi _{1}^{ab}\left( X\right) ^{g%
\acute{e}o}$ where $\pi _{1}^{ab}\left( X\right) ^{g\acute{e}o}$ is
defined by the exact sequence

\begin{equation*}
0\longrightarrow \pi _{1}^{ab}\left( X\right)
^{g\acute{e}o}\longrightarrow
\pi _{1}^{ab}\left( X\right) \longrightarrow Gal(k_{1}^{ab}/k_{1})\mathbb{%
\longrightarrow }02
\end{equation*}

\noindent The most important results in this context are:
\begin{enumerate}
    \item[1)]The quotient of $\pi _{1}^{ab}\left( X\right) $ by the closure
of the
image of $\sigma $ and the cokernel of $\tau $ are both isomorphic to $%
\widehat{\mathbb{Z}}^{r}$ where $r$ is the rank of the curve.
    \item[2)]For this integer $r$, there is an exact sequence
\end{enumerate}

\begin{equation}
0\longrightarrow \left( \mathbb{Q}/\mathbb{Z}\right)
^{r}\longrightarrow H^{3}\left( K,\mathbb{Q}/\mathbb{Z}\left(
2\right) \right) \longrightarrow
\underset{v\in P}{\oplus }\mathbb{Q}/\mathbb{Z\longrightarrow Q}/\mathbb{%
Z\longrightarrow }0  \tag{1.2}
\end{equation}

where $K=K\left( X\right) $ is\ the function field of $X$ and $P$
designates the set of closed points of $X$.

These results are obtained by Saito in [8] generalizing the
previous work of Bloch where he is reduced to the good reduction
case [8, Introduction]. The method of Saito depends on class field
theory for two-dimensional local ring having finite residue field.
He shows these results for general curve except for the $p$
-primary part in char$k=p>0$ case [8,Section II-4]. The remaining
$p$ -primary part had been proved by Yoshida in [11].

There is another direction for proving these results pointed out by
Douai in [3]. It consists to consider for all $l$ prime to the
residual characteristic, the group $Co\ker \sigma $ as the dual of
the group $W_{0}$
of the monodromy weight filtration of $H^{1}(\overline{X},\mathbb{Q}_{\ell }/%
\mathbb{Z}_{\ell })$
\begin{equation*}
H^{1}(\overline{X},\mathbb{Q}_{\ell }/\mathbb{Z}_{\ell
})=W_{2}\supseteq W_{1}\supseteq W_{0}\supseteq 0
\end{equation*}%
where $\overline{X}2=X\otimes _{k_{1}}\overline{k_{1}}$ and
$\overline{k_{1}}$ is an algebraic closure of $k_{1}$. This allow
him to extend the precedent results to projective smooth surfaces
[3].

The aim of this paper is to use a combination of this approach and
the theory of the monodromy-weight filtration of degenerating
abelian varieties on local fields explained by Yoshida in his paper
[11], to study curves over two-dimensional local fields (section 3).

Let $X$ \ be a projective smooth curve defined over two dimensional
local field $k$. Let $K$ \ be its function field and $P$ denotes the
set of closed points of $X.$ For each $\ v\in P$, $k\left( v\right)
$ denotes the residue field at $v\in P. $A finite etale covering $\
Z\rightarrow X$ \ of $X$ is called a c.s covering, if for any closed
point $x$ of $X$, $x\times _{X}Z$ \ is isomorphic to a finite sum of
$\ x$. We denote by $\pi _{1}^{c.s}\left( X\right) $ the quotient
group of $\pi _{1}^{ab}\left( X\right) $ which classifies abelian
c.s coverings of $X$.

\ To study the class field theory of the curve $X$, we construct the
generalized reciprocity map

\begin{equation*}
\sigma /\ell :SK_{2}\left( X\right) /\ell \longrightarrow \pi
_{1}^{ab}\left( X\right) /\ell
\end{equation*}%
where $SK_{2}\left( X\right) /\ell =Co\ker \left\{ K_{3}\left(
K\right)
/\ell \overset{\oplus \partial _{v}}{\longrightarrow }\underset{v\in P}{%
\oplus }K_{2}\left( k\left( v\right) \right) /\ell \right\} $ and $\tau /l$ $%
:$ $V(X)/\ell $ $\longrightarrow \pi _{1}^{ab}\left( X\right) ^{g\acute{e}%
o}/\ell $ for all $\ell $ prime to residual characteristic. The
group $V(X)$ is defined to be the kernel of the norm map
$N:SK_{2}\left( X\right)
\longrightarrow K_{2}(k)$ induced by the norm map $N_{k(v)/k^{x}}:K_{2}%
\left( k(v)\right) \longrightarrow K_{2}(k)$ for all $v$ and $\pi
_{1}^{ab}\left( X\right) ^{g\acute{e}o}$ by the exact sequence

\begin{equation*}
0\longrightarrow \pi _{1}^{ab}\left( X\right)
^{g\acute{e}o}\longrightarrow
\pi _{1}^{ab}\left( X\right) \longrightarrow Gal(k^{ab}/k)\mathbb{%
\longrightarrow }0
\end{equation*}
The cokernel of $\sigma /\ell $ is the quotient group of $\ \pi
_{1}^{ab}\left( X\right) /\ell $ \ that classifies completely split
coverings of $X$ \ ; that is ; $\pi _{1}^{c.s}\left( X\right) /\ell
$.

We begin by proving the exactness of the Kato-Saito sequence
(Proposition 4.2)~:
$$
\begin{array}{lcr}
0\longrightarrow \pi _{1}^{c.s}\left( X\right) /\ell &\longrightarrow &
 H^{4}\left( K,\mathbb{Z}/\mathbb{\ell }\left( 3\right) \right)\qquad\qquad\qquad\\
&\longrightarrow &\underset{v\in P}{\oplus }H^{3}\left( k\left( v\right) ,%
\mathbb{Z}/\mathbb{\ell }\left( 2\right) \right) \longrightarrow \mathbb{Z}/%
\mathbb{\ell \longrightarrow }0
\end{array}
$$
To determinate the group $\pi _{1}^{c.s}\left( X\right) /\ell $, we
need to consider a semi stable model of the curve $X$ ( see Section
5 ) and the weight filtration on its special fiber. In fact, we will
prove in (Proposition 5.1) that $\pi _{1}^{c.s}\left( X\right) \otimes
\mathbb{Q}_{\ell }$ admits a quotient of type $\mathbb{Q}_{l}^{r}$
where $r$ is the rank of the first crane of this filtration.

Now, to investigate the group $\pi _{1}^{ab}\left( X\right)
^{g\acute{e}o},$ we use class field theory of two-dimensional local
field and prove the vanishing of the group $H^{2}\left(
k,\mathbb{Q}/\mathbb{Z}\right) $ (theorem 3.1 ). This yields the
isomorphism

\begin{equation*}
\pi _{1}^{ab}\left( X\right) ^{g\acute{e}o}\simeq \pi
_{1}^{ab}\left( \overline{X}\right) _{G_{k}}
\end{equation*}

Finally, by the Grothendick weight filtration on the group $\pi
_{1}^{ab}\left( \overline{X}\right) _{G_{k}}$ and assuming the
semi-stable reduction, we obtain the structure of the group $\pi
_{1}^{ab}\left(
X\right) ^{g\acute{e}o}$ and information about the map $\tau $ $:$ $V(X)$ $%
\longrightarrow \pi _{1}^{ab}\left( X\right) ^{g\acute{e}o}.$

Our paper is organized as follows. Section 2 is devoted to some
notations. Section 3 contains the proprieties which we need
concerning two-dimensional local field: duality and the vanishing of
the second cohomology group. In section 4, we construct the
generalized reciprocity map and study the Bloch-Ogus complex
associated to $X.$ In section 5, we investigate the group $\pi
_{1}^{c.s}\left( X\right) .$

\section{Notations}

For an abelian group $M$, and a positive integer $n\geq 1,M/n$
denotes the group $M/nM.$

For a scheme $Z,$ and a sheaf $\mathcal{F}$ over the \'{e}tale site of $Z,$ $%
H^{i}\left( Z,\mathcal{F}\right) $ denotes the i-th \'{e}tale
cohomology group. The group $H^{1}\left( Z,\mathbb{Z}/\ell \right) $
is identified with the group of all continues homomorphisms
$\pi_{1}^{ab}\left( Z\right) \longrightarrow \mathbb{Z}/\ell $. If
$\ell$ is invertible on,$\mathbb{Z}/\ell (1)$ denotes the sheaf of
$l$-th root of unity and for any integer $i,$we denote
$\mathbb{Z}/\ell \left( i\right) =\left( \mathbb{\ Z}/\ell \left(
1\right) \right) ^{\otimes i}$

For a field $L$, $K_{i}\left( L\right) $ is the i-th Milnor group.
It
coincides with the $i-$th Quillen group for $i\leq 2.$ For $\ell $ prime to $%
char$ $L$, there is a Galois symbol
\begin{equation*}
h_{\ell ,L}^{i}\,\,\,\,K_{i}L/\ell \longrightarrow H^{i}(L,\mathbb{\
Z}/\ell \left( i\right) )
\end{equation*}%
which is an isomorphism for $i=0,1,2$ ($i=2$ is Merkur'jev-Suslin).

\section{On two-dimensional local field}

A local field $k$ is said to be $n-$dimensional\textit{\ local} if
there exists the following sequence of fields $k_{i}~\left( 1\leq
i\leq n\right) $ such that

\noindent(i) each $k_{i}$ is a complete discrete valuation field having $%
k_{i-1}$ as the residue field of the valuation ring $O_{k_{i}}$ of
$k_{i},$ and

\noindent(ii) $k_{0}$ is a finite field.

For such a field, and for $\ell $ prime to Char($k$), the well-known
isomorphism
\begin{eqnarray}
H^{n+1}\left( k,\mathbb{Z}/\ell \left( n\right) \right) \simeq \mathbb{Z}%
/\ell \qquad \hfill{(3.1)}
\end{eqnarray}%
and for each $i\in \{0,...,n+1\}$ a perfect duality

\begin{equation}
H^{i}( k,\mathbb{Z}/\ell ( j) ) \times H^{n+1-i}(
k,\mathbb{Z}/\ell (n-j) \longrightarrow H^{n+1}( k,\mathbb{Z}%
/\ell ( n) ) \simeq \mathbb{ Z}/\ell  \hfill{(3.2)}
\end{equation}

hold. 8 The class field theory for such fields is summarized as
follows: There is a map $h:$ $K_{2}\left( k\right) $
$\longrightarrow Gal(k^{ab}/k)$ which generalizes the classical
reciprocity map for usually local fields. This map induces an
isomorphism $K_{2}\left( k\right) /N_{L/k}K_{2}\left( L\right)
\simeq Gal(L/k)$ for each finite abelian extension $L$ of $k.$
Furthermore, the canonical pairing

\begin{equation}
H^{1}\left( k,\mathbb{Q}_{l}/\mathbb{Z}_{l}\right) \times
K_{2}(k)\longrightarrow H^{3}\left(
k,\mathbb{Q}_{l}/\mathbb{Z}_{l}\left( 2\right) \right) \simeq
\mathbb{Q}_{l}/\mathbb{Z}_{l}  \tag{3.3}
\end{equation}

induces an injective homomorphism

\begin{equation}
H^{1}\left( k,\mathbb{Q}_{l}/\mathbb{Z}_{l}\right) \longrightarrow
Hom(K_{2}(k),\mathbb{Q}_{l}/\mathbb{Z}_{l})  \tag{3.4}
\end{equation}

It is well-known that the group $H^{2}\left(
M,\mathbb{Q}/\mathbb{Z}\right) $ vanishes when $M$ is a finite field
or usually local field. Next, we prove the same result for
two-dimensional local field

\begin{thm}
If $k$ is a two-dimensional local field of characteristic zero, then
the group $H^{2}\left( k,\mathbb{Q}/\mathbb{Z}\right) $ vanishes.
\end{thm}

\begin{proof}
We proceed as in the proof of theorem 4 of [10]. It is enough to prove that $%
H^{2}\left( k,\mathbb{Q}_{l}/\mathbb{Z}_{l}\right) $ vanishes for
all $l$ and when $k$ contains the group $\mathbb{\mu }_{l}$ of
$l$-th roots of unity. For this, we prove that multiplication by $l$
is injective. That is, we have to show that the coboundary map

\begin{equation*}
H^{1}\left( k,\mathbb{Q}_{l}/\mathbb{Z}_{l}\right) \overset{\delta }{%
\longrightarrow }H^{2}\left( k,\mathbb{Z}/l\mathbb{Z}\right)
\end{equation*}%
is injective.

By assumption on $k$, we have

\begin{equation*}
H^{2}\left( k,\mathbb{Z}/l\mathbb{Z}\right) \simeq H^{2}\left(
k,\mathbb{\mu }_{l}\right) \simeq \mathbb{Z}/\ell
\end{equation*}

The last isomorphism is well-known for one-dimensional local field
and was generalized to non archimedian and locally compact fields by
Shatz in [6]. The proof is now reduced to the fact that $\delta \neq
0;$

By class field theory of two dimensional local field, the cohomology group
$H^{1}\left( k,\mathbb{Q}_{l}/\mathbb{Z}_{l}\right) $ may be
identified with the group of
continuous homomorphisms $K_{2}(k)\overset{\Phi }{\longrightarrow }\mathbb{Q}%
_{l}/\mathbb{Z}_{l}$.

Now, $\delta (\Phi )=0$ if and only if \ $\Phi $ is a $l-$th power, and $%
\Phi $ is a $l-$th power if and only if $\Phi $ is trivial on $\mathbb{\mu }%
_{l}$. Thus, it is sufficient to construct an homomorphism $%
K_{2}(k)\longrightarrow \mathbb{Q}_{l}/\mathbb{Z}_{l}$ which is non
trivial on8 $\mathbb{\mu }_{l}.$

Let $i$ be the maximal natural number such that $k$ contains a primitive $%
l^{i}-$th root of unity. Then, the image $\xi $ of a primitive
$l^{i}-$th root of unity under the composite map

\begin{equation*}
k^{x}/k^{xl}\simeq H^{1}\left( k,\mathbb{\mu }_{l}\right) \simeq
H^{1}\left(
k,\mathbb{Z}/l\mathbb{Z}\right) \longrightarrow H^{1}\left( k,\mathbb{Q}_{l}/%
\mathbb{Z}_{l}\right)
\end{equation*}

is not zero. Thus, the injectivity of the map

\begin{equation*}
H^{1}\left( k,\mathbb{Q}_{l}/\mathbb{Z}_{l}\right) \longrightarrow
Hom(K_{2}(k),\mathbb{Q}_{l}/\mathbb{Z}_{l})
\end{equation*}

gives rise to a character which is non trivial on $\mathbb{\mu
}_{l}.$
\end{proof}

\section{Curves over two dimensional local field}

Let $k$ be a two dimensional local field of characteristic zero and
$X$ a smooth projective curve defined over $k.$

\noindent We recall that we deno8te:

\noindent $K=K\left( X\right) $ \ its function field,

\noindent $P:$ set of closed points of $X$, and for $v\in P,$

$k\left( v\right) :$ the residue field at $v\in P$

The residue field of $k$ is one-dimensi6nal local field. It is denoted by $%
k_{1}$

\noindent Let $\mathcal{H}^{n}\left( \mathbb{Z}/\ell \left( 3\right)
\right) $ $,n\geq 1$, the Zariskien sheaf associated to the presheaf
$\ U\longrightarrow H^{n}\left( U,\mathbb{Z}/\ell \left( 3\right)
\right) $. Its cohomology is calculated by the Bloch-Ogus
resolution. So, we have the two exact sequences:

\begin{equation}
H^{3}\left( K,\mathbb{Z}/\ell\left( 3\right) \right) \longrightarrow
\underset{v\in P}{\oplus}H^{2}\left( k\left( v\right)
,\mathbb{Z}/\ell\left(
2\right) \right) \longrightarrow H^{1}\left( X_{Zar},\mathcal{H}^{3}(\mathbb{%
Z}/\ell\left( 3\right) \right) )\longrightarrow 0 
\end{equation}

\begin{equation}
0\longrightarrow H^{0}( X_{Zar},\mathcal{H}^{4}(\mathbb{Z}/\ell(3))
)\longrightarrow H^{4}( K,\mathbb{Z}/\ell( 3) )
\longrightarrow\underset{v\in P}{\oplus}H^{3}( k(v,\mathbb{Z}/\ell(2) )) 
\end{equation}

\subsection{The reciprocity map}

We introduce the group $SK_{2}\left( X\right) /\ell :$

\begin{equation*}
SK_{2}\left( X\right) /\ell=Co\ker\left\{ K_{3}\left( K\right) /\ell\overset{%
\oplus\partial_{v}}{\longrightarrow}\underset{v\in P}{\oplus
}K_{2}\left( k\left( v\right) \right) /\ell\right\}
\end{equation*}

\noindent where $\partial _{v}:K_{3}\left( K\right) \longrightarrow
K_{2}\left( k\left( v\right) \right) $ \ is the boundary map in
K-Theory. It will play an important role in class field theory for
$X$ as pointed out by Saito in the introduction of [8]. In this
section, we construct a map

\begin{equation*}
\sigma/\ell:SK_{2}\left( X\right)
/\ell\longrightarrow\pi_{1}^{ab}\left( X\right) /\ell
\end{equation*}

\noindent which describe the class field theory of $X$.

By definition of $SK_{2}\left( X\right) /\ell ,$ we have the exact
sequence

\begin{equation*}
K_{3}\left( K\right) /\ell\longrightarrow\underset{v\in P}{\oplus}%
K_{2}\left( k\left( v\right) \right) /\ell\longrightarrow
SK_{2}\left( X\right) /\ell\longrightarrow0
\end{equation*}

On the other hand, it is known that the following diagram is
commutative:

\begin{equation*}
\begin{array}{ccc}
K_{3}\left( K\right) /\ell & \longrightarrow\underset{v\in
P}{\oplus} &
K_{2}\left( k\left( v\right) \right) /\ell \\
\downarrow h^{3} &  & \downarrow h^{2} \\
H^{3}\left( K,\mathbb{Z}/\ell\left( 3\right) \right) &
\longrightarrow
\underset{v\in P}{\oplus} & H^{2}\left( k\left( v\right) ,\mathbb{Z}%
/\ell\left( 2\right) \right)%
\end{array}%
\end{equation*}

\noindent where $h^{2},h^{3}$ are the Galois symbols. This yields
the existence of a morphism

\begin{equation*}
h:SK_{2}\left( X\right) /\ell\longrightarrow H^{1}\left( X_{Zar},\mathcal{H}%
^{3}(\mathbb{Z}/\ell\left( 2\right) \right) )
\end{equation*}

\noindent taking in account the exact sequence (4.1). This morphism
fit in the following commutative diagram

\begin{equation*}
\begin{array}{ccccccc}
0\longrightarrow & K_{3}\left( K\right) /\ell & \longrightarrow \,\underset{%
v\in P}{\oplus } & K_{2}\left( k\left( v\right) \right) /\ell &
\rightarrow
& \,SK_{2}(X)/\ell & \longrightarrow 0 \\
& \downarrow h^{3} &  & \downarrow h^{2} &  & \downarrow h &  \\
0\longrightarrow & H^{3}\left( K,\mathbb{Z}/\ell \left( 2\right)
\right) &
\longrightarrow \underset{v\in P}{\oplus } & H^{2}\left( k\left( v\right) ,%
\mathbb{Z}/\ell \left( 2\right) \right) & \rightarrow & H^{1}\left( X_{Zar},%
\mathcal{H}^{3}(\mathbb{Z}/\ell \left( 2\right) \right) ) & \longrightarrow 0%
\end{array}%
\end{equation*}%
\noindent

\noindent

By Merkur'jev-Suslin, the map $h^{2}$ is an isomorphism, which imply
that $h$ is surjective. On the other hand the spectral sequence

\begin{equation*}
H^{p}\left( X_{Zar},\mathcal{H}^{q}(\mathbb{Z}/\ell\left( 3\right)
\right) )\Rightarrow H^{p+q}(X,\mathbb{Z}/\ell\left( 3\right) )
\end{equation*}

\noindent induces the exact sequence

\begin{align}
0 & \longrightarrow H^{1}\left( X_{Zar},\mathcal{H}^{3}(\mathbb{Z}%
/\ell\left( 3\right) \right) )\overset{e}{\longrightarrow}H^{4}(X,\mathbb{Z}%
/\ell\left( 3\right) )  \tag{4.3} \\
& \longrightarrow H^{0}\left(
X_{Zar},\mathcal{H}^{4}(\mathbb{Z}/\ell\left(
3\right) \right) )\longrightarrow H^{2}\left( X_{Zar},\mathcal{H}^{3}(%
\mathbb{Z}/\ell\left( 3\right) \right) )=0  \notag
\end{align}

\noindent Composing $h$ and $e$, we get the map

\begin{equation*}
SK_{2}\left( X\right) /\ell\longrightarrow
H^{4}(X,\mathbb{Z}/\ell\left( 3\right) )
\end{equation*}

\noindent

\noindent Finally the group $H^{4}(X,\mathbb{Z}/\ell \left( 3\right)
)$ \ is identified to the group $\pi _{1}^{ab}\left( X\right) /\ell
$ by the duality [4,II, th 2.1]

\begin{equation*}
H^{4}(X,\mathbb{Z}/\ell \left( 3\right) )\otimes
H^{1}(X,\mathbb{Z}/\ell
)\longrightarrow H^{5}(X,\mathbb{Z}/\ell \left( 3\right) )\simeq H^{3}(k,%
\mathbb{Z}/\ell \left( 2\right) )\simeq \mathbb{Z}/\ell  
\end{equation*}
Hence, we obtain the map

\begin{equation*}
\sigma/\ell:SK_{2}\left( X\right)
/\ell\longrightarrow\pi_{1}^{ab}\left( X\right) /\ell
\end{equation*}

\noindent

\begin{rem}
By the exact sequence (4.2) the group $H^{0}\left( X_{Zar},\mathcal{H}^{4}(%
\mathbb{Z}/\ell \left( 3\right) \right) )$ \ coincides with the
kernel of the map2
\end{rem}

\begin{equation*}
H^{4}(K,\mathbb{Z}/\ell\left( 3\right) )\longrightarrow\underset{v\in P}{%
\oplus}H^{3}\left( k\left( v\right) ,\mathbb{Z}/\ell\left( 2\right)
\right)
\end{equation*}

\noindent {\textit{and by localization in \'{e}tale cohomology}}

\begin{equation*}
\underset{v\in P}{\oplus }\mathit{H}^{2}\left( k\left( v\right) ,\mathbb{Z}%
/\ell \left( 2\right) \right) \mathit{\longrightarrow H}^{4}\left( X,\mathbb{%
Z}/\ell \left( 3\right) \right) \mathit{\longrightarrow H}^{4}\left( K,%
\mathbb{Z}/\ell \left( 3\right) \right) \underset{v\in
P}{\longrightarrow \oplus }\mathit{H}^{3}\left( k\left( v\right)
,\mathbb{Z}/\ell \left( 2\right) \right)
\end{equation*}

\textit{\noindent and taking in account (4.3), we see that
}$H^{1}\left( X_{Zar},\mathcal{H}^{4}(\mathbb{Z}/\ell \left(
3\right) \right) )$\textit{\ \ is the cokernel of the Gysin map}

\begin{equation*}
\underset{v\in P}{\oplus }\mathit{H}^{2}\left( k\left( v\right) ,\mathbb{Z}%
/\ell \left( 2\right) \right) \overset{g}{\longrightarrow }\mathit{H}%
^{4}\left( X,\mathbb{Z}/\ell \left( 3\right) \right)
\end{equation*}

\textit{\noindent and consequently the morphism }$g$\textit{\
factorize through\ }$H^{1}\left(
X_{Zar},\mathcal{H}^{4}(\mathbb{Z}/\ell \left( 3\right) \right) )$

\begin{equation*}
\begin{array}{ccc}
\underset{v\in P}{\oplus }\mathit{H}^{2}\left( k\left( v\right) ,\mathbb{Z}%
/\ell \left( 2\right) \right) & \overset{g}{\longrightarrow } & \mathit{H}%
^{4}\left( X,\mathbb{Z}/\ell \left( 3\right) \right) \\
\mathit{\searrow } &  & \mathit{\nearrow } \\
& \mathit{H}^{1}\left( X_{Zar},\mathcal{H}^{4}(\mathbb{Z}/\ell
\left( 3\right) \right) \mathit{)} &
\end{array}%
\end{equation*}

\textit{\noindent Then, we deduce the following commutative diagram}

\begin{equation*}
\begin{array}{ccccc}
K_{3}\left( K\right) /\ell & \rightarrow \underset{v\in P}{\oplus }
& K_{2}(k\left( v\right) )/\ell \, & \rightarrow & SK_{2}\left(
X\right) /\ell
\longrightarrow 0 \\
\downarrow h^{3} &  & \downarrow h^{2} &  & \downarrow h \\
H^{3}\left( K,\mathbb{Z}/\ell \left( 3\right) \right) & \rightarrow \underset%
{v\in P}{\oplus } & H^{2}\left( k\left( v\right) ,\mathbb{Z}/\ell
\left(
2\right) \right) & \rightarrow & H^{1}\left( X_{Zar},\mathcal{H}^{4}(\mathbb{%
Z}/\ell \left( 3\right) \right) )\longrightarrow 0 \\
&  & \downarrow g & \swarrow e &  \\
&  & \pi _{1}^{ab}\left( X\right) /l=H^{4}\left( X,\mathbb{Z}/\ell
\left( 3\right) \right) &  &
\end{array}%
\end{equation*}%
\textit{The surjectivity of the map }$h$\textit{\ implies that the
cokernel of }$\ \ \ \ \ \ \ \ \ \ \ \ \ \ \ \ \ \ \ \ \ \ \ \ \ \ \
\ \ \ \ \ \ \ \ \
\ \ \ \ $%
\begin{equation*}
\sigma /\ell :SK_{2}\left( X\right) /\ell \longrightarrow \pi
_{1}^{ab}\left( X\right) /\ell
\end{equation*}%
$\ $\textit{coincides with the cokernel of }$e$\textit{\ which is }$%
H^{0}\left( X_{Zar},\mathcal{H}^{4}(\mathbb{Z}/\ell \left( 3\right)
\right) ).$\textit{\ Hence }$Co\ker \sigma /\ell $\textit{\ is the
dual of the
kernel of the map }%

\begin{equation}
H^{1}\left( X,\mathbb{Z}/\ell \right) \longrightarrow
\mathop{\displaystyle \prod }\limits_{v\in P}H^{1}\left( k\left(
v\right) ,\mathbb{Z}/\ell \right) \tag{4.4}
\end{equation}%

\subsection{The Kato-Saito exact sequence}

\begin{defn}
Let $Z$ be a Noetherian scheme. A finite etale covering $f: W
\rightarrow Z$ is called a c.s covering if for any closed point
$z$ of $Z$ , $z\times _{Z}W$ is isomorphic to a finite
scheme-theoretic sum of copies of $z$ We denote $\pi
_{1}^{c.s}\left( Z\right) $ the quotient group of $\pi
_{1}^{ab}\left( Z\right) $  which classifies abelian c.s coverings
of $Z.$
\end{defn}

\bigskip

Hence, the group $\pi _{1}^{c.s}\left( X\right) /\ell $ is the dual
of the kernel of the map
\begin{equation}
H^{1}\left(
X,\mathbb{Z}/\ell \right) \longrightarrow \mathop{\displaystyle
\prod }\limits_{v\in P}H^{1}\left( k\left( v\right) ,\mathbb{Z}/\ell
\right) \tag{4.4}
\end{equation}
as in [8, section 2, definition and sentence just below]. Now, we
are able to calculate the homologies of the Bloch-Ogus complex
associated $\ X.$

\noindent\ \ \ \ \ Generalizing [9,Theorem7], we obtain~:
\begin{prop}
Let $X$ be a projective smooth curve defined over
$k$
Then for all $\ell $, we have the following exact sequence
$$
\begin{array}{lcr}
0\longrightarrow \pi _{1}^{c.s}\left( X\right) /\ell &\longrightarrow &
 H^{4}\left( K,\mathbb{Z}/\mathbb{\ell }\left( 3\right) \right)\qquad\qquad\qquad\\
&\longrightarrow &\underset{v\in P}{\oplus }H^{3}\left( k\left( v\right) ,%
\mathbb{Z}/\mathbb{\ell }\left( 2\right) \right) \longrightarrow \mathbb{Z}/%
\mathbb{\ell \longrightarrow }0.
\end{array}
$$

\end{prop}
\begin{proof}
Consider the localization sequence on $X$
\begin{equation*}
\begin{array}{c}
\underset{v\in P}{\oplus }H^{2}\left( k\left( v\right)
,\mathbb{Z}/\ell
\left( 2\right) \right) \overset{g}{\longrightarrow }H^{4}\left( X,\mathbb{Z}%
/\ell \left( 3\right) \right) \longrightarrow H^{4}\left(
K,\mathbb{Z}/\ell
\left( 3\right) \right) \\
\longrightarrow \underset{v\in P}{\oplus }H^{3}\left( k\left( v\right) ,%
\mathbb{Z}/\ell \left( 2\right) \right) \longrightarrow H^{5}\left( X,%
\mathbb{Z}/\ell \left( 3\right) \right) \longrightarrow 0%
\end{array}
\end{equation*}
We know that the cokernel of the Gysin map $g$ coincides with
$\pi_{1}^{c.s}\left( X\right) /\ell $ and we use the isomorphism
$H^{5}\left( X,\mathbb{Z}/\ell \left( 3\right) \right) \simeq \mathbb{Z}/\ell $
(4.4).
\end{proof}

\section{The group $\protect\pi _{1}^{c.s}\left( X\right) $}

In his paper [8], Saito don't prove the $p-$ primary part in the char $%
k=p\gtrdot 0$ case. This case was developed by Yoshida in [11]. His
method is based on the theory of monodromy-weight filtration of
degenerating abelian varieties on local fields. In this work, we use
this approach to investigate the group $\pi _{1}^{c.s}\left(
X\right) .$ As mentioned by Yoshida in [11,section 2] Grothendieck's
theory of monodromy-weight filtration on Tate module of abelian
varieties are valid where the residue field is arbitrary perfect
field$\,$

We assume the semi-stable reduction and choose a regular model
$\mathcal{X}$ of $X$ over $SpecO_{k},$ by which we mean a two
dimensional regular scheme with a proper birational morphism
$f:\mathcal{X}$ $\longrightarrow SpecO_{k}$ such that $\mathcal{X}$
$\otimes _{O_{k}}k\simeq X$ and if $\mathcal{X}_{s}$
designates the special fiber $\mathcal{X}$ $\otimes _{O_{k}}k_{1},$ then $Y=(%
\mathcal{X}_{s})_{r\acute{e}d}$ is a curve defined over the residue field $%
k_{1}$ such that any irreducible component of $Y$ is regular and it
has ordinary double points as singularity.

Let $\overline{Y}=Y\otimes _{k_{1}}\overline{k_{1}}$ , where $\overline{k_{1}%
}$ is an algebraic closure of $k_{1}$ and $\overline{Y}^{[p]}=\!\!\!\!\!%
\underset{i_{/}<i_{1}<\cdots <i_{p}}{\mathop{\displaystyle \bigsqcup } }\!\!%
\overline{Y_{i_{/}}}\cap \overline{Y_{i_{1}}}\cap \cdots \cap \overline{%
Y_{i_{p}}}$ $,(\overline{Y})_{i\in I}=$ collection of irreducible
components of $\overline{Y}.$

\bigskip Let $\left\vert \overline{\Gamma }\right\vert $ be a realization of
the dual graph $\overline{\Gamma },$ then the group $H^{1}\left(
\left\vert \overline{\Gamma }\right\vert ,\mathbb{Q}_{l}\right) $
coincides with the group $W_{0}(H^{1}\left(
\overline{Y},\mathbb{Q}_{l}\right) $ $)$ constituted of elements of
weight $0$ for the filtration

\begin{equation*}
H^{1}(\overline{Y},\mathbb{Q}_{\ell })=W_{1}\supseteq W_{0}\supseteq
0
\end{equation*}

of $H^{1}(\overline{Y},\mathbb{Q}_{\ell })~$deduced from the
spectral sequence$\ \ \ $

$\ $%
\begin{equation*}
E_{1}^{p,q}=H^{q}(\overline{Y}^{[p]},\mathbb{Q}_{\ell
})\Longrightarrow H^{p+q}(\overline{Y},\mathbb{Q}_{\ell })
\end{equation*}

For details see [2], [3] and [5]

$\ $

Now, if we assume further that the irreducible components and double
points
of $\overline{Y}$ are defined over $k_{1},$ then the dual graph $\overline{%
\Gamma }$ of $\overline{Y}$ go down to $k_{1}$ and we obtain the
injection

\begin{equation*}
W_{0}(H^{1}\left( \overline{Y},\mathbb{Q}_{l}\right) )\subseteq
H^{1}\left( Y,\mathbb{Q}_{l}\right) \hookrightarrow H^{1}\left(
X,\mathbb{Q}_{l}\right)
\end{equation*}

\begin{prop}
The group $\pi _{1}^{c.s}\left( X\right) \otimes $
$\mathbb{Q}_{l}~$admits a quotient of type
$\mathbb{Q}_{l}^{r},~$where $r$ is the $\mathbb{Q}_{l}-rank$
of \ the group $H^{1}\left( \left\vert \overline{\Gamma }\right\vert ,%
\mathbb{Q}_{l}\right) $
\end{prop}

\begin{proof}
We know (4.5) that $\pi _{1}^{c.s}\left( X\right) \otimes $
$\mathbb{Q}_{l}$
is the dual of the kernel of the map%
\begin{equation*}
\alpha :H^{1}\left( X,\mathbb{Q}_{l}\right) \longrightarrow %
\mathop{\displaystyle \prod }\limits_{v\in P}H^{1}\left( k\left( v\right) ,%
\mathbb{Q}_{l}\right)
\end{equation*}%
We will prove that $W_{0}(H^{1}\left(
\overline{Y},\mathbb{Q}_{l}\right)
)\subseteq Ker\alpha .$ The group $W_{0}=W_{0}(H^{1}\left( \overline{Y},%
\mathbb{Q}_{l}\right) )$ is calculated as the homologie of the
complex

\begin{equation*}
H^{0}(\overline{Y}^{[0]},\mathbb{Q}_{\ell })\longrightarrow H^{0}(\overline{Y%
}^{[1]},\mathbb{Q}_{\ell })\longrightarrow 0
\end{equation*}%
Hence $W_{0}=$ $H^{0}(\overline{Y}^{[1]},\mathbb{Q}_{\ell })/\mathop{\rm Im}%
\{H^{0}(\overline{Y}^{[0]},\mathbb{Q}_{\ell })\longrightarrow H^{0}(%
\overline{Y}^{[1]},\mathbb{Q}_{\ell })\}.$ Thus, it suffices to
prove the vanishing of the composing map

$H^{0}(\overline{Y}^{[1]},\mathbb{Q}_{\ell })\longrightarrow
W_{0}\subseteq
H^{1}\left( Y,\mathbb{Q}_{l}\right) \hookrightarrow H^{1}\left( X,\mathbb{Q}%
_{l}\right) \longrightarrow H^{1}\left( k\left( v\right) ,\mathbb{Q}%
_{l}\right) $

for all $v\in P.$

Let $z_{v}$ be the $0-$ cycle in $\overline{Y}$ obtained by
specializing $v,$ which induces a map $z_{v}^{[1]}\longrightarrow
\overline{Y}^{[1]}.$ Co8nsequently, the map
$H^{0}(\overline{Y}^{[1]},\mathbb{Q}_{\ell })\longrightarrow
H^{1}\left( k\left( v\right) ,\mathbb{Q}_{l}\right) $ factors as
follows

\begin{equation*}
\begin{array}{ccc}
H^{0}(\overline{Y}^{[1]},\mathbb{Q}_{\ell }) & \longrightarrow &
H^{1}\left(
k\left( v\right) ,\mathbb{Q}_{l}\right) \\
\searrow &  & \nearrow \\
& H^{0}(z_{v}^{[1]},\mathbb{Q}_{\ell }) &
\end{array}
\end{equation*}

But the trace $z_{v}^{[1]}$ of $\overline{Y}^{[1]}$ on $z_{v}$ is
empty. This implies the vanishing of
$H^{0}(z_{v}^{[1]},\mathbb{Q}_{\ell }).$
\end{proof}

\bigskip

Let $V(X)$ be the kernel of the norm map $N:SK_{2}\left( X\right)
\longrightarrow K_{2}(k)$ induced by the norm map $N_{k(v)/k^{x}}:K_{2}%
\left( k(v)\right) \longrightarrow K_{2}(k)$ for all $v$ . Then, we
obtain a map $\tau /l$ $:$ $V(X)/\ell $ $\longrightarrow \pi
_{1}^{ab}\left( X\right) ^{g\acute{e}o}/\ell $ and a commutative
diagram

\begin{equation*}
\begin{array}{ccccc}
V(X)/\ell & \longrightarrow \, & SK_{2}\left( X\right) /\ell &
\rightarrow &
K_{2}(k)/\ell \\
\downarrow \tau /l &  & \downarrow \sigma /\ell &  & \downarrow h/l \\
\pi _{1}^{ab}\left( X\right) ^{g\acute{e}o}/\ell & \longrightarrow &
\pi
_{1}^{ab}\left( X\right) /\ell & \rightarrow & Gal(k^{ab}/k)/l%
\end{array}%
\end{equation*}

where the map $h/l:$ $K_{2}\left( k\right) /l$ $\longrightarrow
Gal(k^{ab}/k)/l$ is the one obtained by class field theory of $k$
(section 3). From this diagram we see that the group $Co\ker \tau
/l$ is isomorphic to the group $Co\ker \sigma /\ell .$ Next, we
investigate the map $\tau /l.$

We begin by the following result which is a consequence of the
structure of the two-dimensional local field $k$

\begin{lem}
There is an isomrphism%
\begin{equation*}
\pi _{1}^{ab}\left( X\right) ^{g\acute{e}o}\simeq \pi
_{1}^{ab}\left( \overline{X}\right) _{G_{k}},
\end{equation*}
where $\pi _{1}^{ab}\left( \overline{X}\right) _{G_{k}}$ is the
group of coinvariants under $G_{k}=Gal(k^{ab}/k)$.
\end{lem}
\begin{proof}
$\bigskip $As in the proof of Lemma 4.3 of [11], this is an
immediate consequence of (Theorem 3.1).
\end{proof}
Finally, we are able to deduce the structure of the group $\pi
_{1}^{ab}\left( X\right) ^{g\acute{e}o}$

\begin{thm}
$\bigskip $The group $\pi _{1}^{ab}\left( X\right)
^{g\acute{e}o}\otimes
\mathbb{Q}_{l}$ is isomorphic to $\widehat{\mathbb{Q}_{l}}^{r}$ and the map $%
\tau :$ $V(X)\longrightarrow \pi _{1}^{ab}\left( X\right)
^{g\acute{e}o}$ is a surjection onto $(\pi _{1}^{ab}\left( X\right)
^{g\acute{e}o})_{tor}$.
\end{thm}

\begin{proof}
By the preceding lemma, we have the isomorphism $\pi _{1}^{ab}\left(
X\right) ^{g\acute{e}o}\simeq \pi _{1}^{ab}\left(
\overline{X}\right)
_{G_{k}}.$ On the other hand the group $\pi _{1}^{ab}\left( \overline{X}%
\right) _{G_{k}}\otimes \mathbb{Q}_{\ell }$ admits the filtration
[12,Lemma 4.1 and section 2]

\begin{equation*}
W_{0}(\pi _{1}^{ab}\left( \overline{X}\right) _{G_{k}}\otimes \mathbb{Q}%
_{l})=\pi _{1}^{ab}\left( \overline{X}\right) _{G_{k}}\otimes \mathbb{Q}%
_{l}\supseteq W_{-1}(\pi _{1}^{ab}\left( \overline{X}\right)
_{G_{k}}\otimes \mathbb{Q}_{l})\supseteq W_{-2}(\pi _{1}^{ab}\left(
\overline{X}\right) _{G_{k}}\otimes \mathbb{Q}_{l})
\end{equation*}%
But; by assumption; the curve $X$ admits a semi-stable reduction,
then the group $Gr_{0}(\pi _{1}^{ab}\left( \overline{X}\right)
_{G_{k}}\otimes \mathbb{Q}_{l})=W_{0}(\pi _{1}^{ab}\left(
\overline{X}\right)
_{G_{k}}\otimes \mathbb{Q}_{l})/W_{-1}(\pi _{1}^{ab}\left( \overline{X}%
\right) _{G_{k}}\otimes \mathbb{Q}_{l})$ has the following structure%
\begin{equation*}
0\longrightarrow Gr_{0}(\pi _{1}^{ab}\left( \overline{X}\right)
_{G_{k}}\otimes \mathbb{Q}_{l})_{tor}\longrightarrow Gr_{0}(\pi
_{1}^{ab}\left( \overline{X}\right) _{G_{k}}\otimes \mathbb{Q}%
_{l})\longrightarrow \widehat{\mathbb{Q}_{l}}^{r^{\prime
}}\longrightarrow 0
\end{equation*}%
where $r^{\prime }$ is the $k-rank$ of $X$. This is confirmed by
Yoshida [11, section 2], independently of the finitude of the
residue field of $k$
considered in his paper. The integer $r^{\prime }$ is equal to the integer $%
r=H^{1}\left( \left\vert \overline{\Gamma }\right\vert ,\mathbb{Q}%
_{l}\right) =H^{1}\left( \left\vert \Gamma \right\vert ,\mathbb{Q}%
_{l}\right) $ by assuming that the irreducible components and double
points of $\overline{Y}$ are defined over $k_{1},$

On the other hand, the exact sequence%
\begin{equation*}
0\longrightarrow W_{-1}(\pi _{1}^{ab}\left( \overline{X}\right)
_{G_{k}})\longrightarrow \pi _{1}^{ab}\left( \overline{X}\right)
_{G_{k}}\longrightarrow Gr_{0}(\pi _{1}^{ab}\left(
\overline{X}\right) _{G_{k}})\longrightarrow 0
\end{equation*}%
and (Proposition 5.1) allow us to conclude that the group $W_{-1}(\pi
_{1}^{ab}\left( \overline{X}\right) _{G_{k}})$ is finite and the map
$\tau :$ $V(X)\longrightarrow \pi _{1}^{ab}\left( X\right)
^{g\acute{e}o}$ is a surjection onto $(\pi _{1}^{ab}\left( X\right)
^{g\acute{e}o})_{tor}$ as established by Yoshida [11]
\end{proof}

\begin{rem}
\bigskip If we apply the same method of Saito to study curves over
two-dimensional local fields, we need class field theory of
two-dimensional local ring having one-dimensional local field as
residue field. This is done by myself in [1]. Hence, one can follow
Saito 's method to obtain the same results.
\end{rem}

\end{document}